\documentclass[11pt]{amsart}

\usepackage{amsmath,amsthm, amscd, amssymb, amsfonts}
\usepackage[dvips, dvipsnames, usenames]{color}
\usepackage{eucal}
\usepackage[all]{xy}
\usepackage{hhline}
\usepackage{verbatim}
\usepackage{longtable}
\usepackage[T1]{fontenc}

\newcommand{\Cent}{\operatorname{C}}

\newcommand{\Imm}{\operatorname{Im}}

\newcommand{\ord}{\operatorname{ord}}

\newcommand\toba{{\mathfrak B }}

\newcommand{\trid}{\triangleright}

\newcommand{\Z}{{\mathbb Z}}
\newcommand{\N}{{\mathbb N}}

\newcommand{\C}{{\mathcal C}}
\newcommand{\cl}{{\mathcal C}_{\ell}}
\newcommand{\D}{{\mathcal D}}

\newcommand{\q}{{\mathbf q}}

\newcommand{\oc}{{\mathcal O}}

\newcommand{\Aut}{\operatorname{Aut}}
\newcommand{\Out}{\operatorname{Out}}

\theoremstyle{plain}

\newtheorem{lema}{Lemma}[section]
\newtheorem{theorem}[lema]{Theorem}
\newtheorem{cor}[lema]{Corollary}

\newtheorem{prop}[lema]{Proposition}

\newtheorem{question}{Question}

\theoremstyle{definition}

\theoremstyle{remark}
\newtheorem{obs}[lema]{Remark}

\newcommand\id{\operatorname{id}}

\newcommand\dn{\mathbb D_n}
\newcommand\an{\mathbb A_n}

\newcommand\A{\mathbb A}

\newcommand\s{\mathbb S}

\def\pf{\begin{proof}}
\def\epf{\end{proof}}

\theoremstyle{remark}

\begin{document}

\renewcommand{\baselinestretch}{1.2}

\thispagestyle{empty}

\title[Twisted homogeneous racks of type D]
{On Twisted homogeneous racks of type D}

\author[Andruskiewitsch, Fantino, Garc\'\i a, Vendramin]{N. Andruskiewitsch,
F. Fantino, G. A. Garc\'\i a, L. Vendramin}

\address{\noindent N. A., F. F., G. A. G.: Facultad de Matem\'atica, Astronom\'{\i}a y F\'{\i}sica,
Universidad Nacional de C\'ordoba. CIEM -- CONICET. 
Medina Allende s/n (5000) Ciudad Universitaria, C\'ordoba,
Argentina}
\address{\noindent F. F., G. A. G.: Facultad de Ciencias Exactas,
F\'isicas y Naturales, Universidad Nacional de C\'ordoba. Velez
Sarsfield 1611 (5000) Ciudad Universitaria, C\'ordoba, Argentina}
\address{\noindent L. V. : Departamento de Matem\'atica -- FCEyN,
Universidad de Buenos Aires, Pab. I -- Ciudad Universitaria (1428)
Buenos Aires -- Argentina}
\address{\noindent L. V. : Instituto de Ciencias, Universidad de Gral. Sarmiento, J.M. Gutierrez
1150, Los Polvorines (1653), Buenos Aires -- Argentina}
\address{}

\email{(andrus, fantino, ggarcia)@famaf.unc.edu.ar} \email{lvendramin@dm.uba.ar}

\thanks{This work was partially supported by ANPCyT-Foncyt, CONICET, Ministerio de Ciencia y
Tecnolog\'{\i}a (C\'ordoba), Secyt-UNC and Secyt-UBA}

\date{\today}

\begin{abstract}
We develop some techniques to check when a twisted homogeneous
rack of class $(L,t,\theta)$ is of type D. Then we apply the
obtained results to the cases $L=\A_n$, $n\geq 5$, or $L$ a
sporadic group.
\end{abstract}

\maketitle

\begin{center}
\emph{Dedicado a Hans-J\"urgen Schneider en ocasi\'on de sus 65
años.}
\end{center}

\section{Introduction}

A \emph{rack} is a pair $(X,\trid)$ where $X$ is a non-empty set
and $\trid:X\times X\to X$ is an operation such that the map
$\varphi_x = x\trid \underline{\quad}$ is bijective for any $x\in
X$, and $x \trid(y\trid z) = (x\trid y) \trid(x\trid z)$ for all
$x,y,z\in X$. For instance, a conjugacy class in a group is a rack
with $\trid$ given by conjugation. Racks appear naturally in the
classification problem of finite-dimensional pointed Hopf
algebras, see \cite{AG-adv}. In this sense, our general aim is to
solve the following question. We omit the definitions of
\emph{cocycle over a rack} and \emph{Nichols algebra
$\toba(X,\q)$} because these are not needed in the present paper,
an omission justified by Theorem \ref{th:racks-typeD} below.

\begin{question}\label{que:racks}
For any finite  rack $X$, and for any cocycle $\q$, determine if $\dim \toba(X, \q)
< \infty$.
\end{question}

\medbreak A rack  $X$ is \emph{of type D} if there exists  a
    decomposable subrack $Y = R\coprod S$ of $X$ such that
    \begin{equation}\label{eqn:hypothesis-subrack}
        r\trid(s\trid(r\trid s)) \neq s, \quad \text{for some } r\in R, s\in S.
    \end{equation}

The reduction of certain problems in the classification of
finite-dimensional pointed Hopf algebras to the study of racks of
type D is given by the following result, a reinterpretation of
\cite[Thm. 8.6]{HS1} in turn a consequence of \cite{AHS}.

\begin{theorem}\label{th:racks-typeD} \cite[Th. 3.6]{AFGV}
    If $X$ is a finite rack of type D, then $X$  collapses, that is $\dim\toba(X,\q)
=\infty$ for any $\q$. \qed
\end{theorem}

Therefore, it is very important to determine all simple racks of type D.
Indeed, any indecomposable rack $Z$ admits a rack
epimorphism $\pi: Z \to X$ with $X$ simple; and
if $X$ is of type D, then $Z$ is of type D.

\medbreak Now, finite simple racks have been classified in
\cite[3.9, 3.12]{AG-adv}, see also \cite{jo}. Succinctly, any
simple rack belongs to one of the following classes:

\begin{enumerate}

    \item Simple affine racks.

    \item Non-trivial (twisted) conjugacy classes in simple groups.

    \item\label{item:twisted-homogeneous} \emph{Simple twisted homogeneous racks} of class $(L,t, \theta)$; these
    are twisted conjugacy classes corresponding to $(G,u)$, where
\end{enumerate}

    \begin{itemize}
      \item $G = L^t$, with $L$ a finite simple non-abelian group, $t > 1$ and  $\theta\in\Aut(L)$,
      \item $u\in \Aut (L^t)$ acts by
        $$u(\ell_1,\ldots,\ell_t)=(\theta(\ell_t),\ell_1,\ldots,\ell_{t-1}), \quad \ell_1,\ldots,\ell_t \in L.$$
    \end{itemize}
 See Subsection \ref{subect:thr} below.

\medbreak The goal of the present paper is to study when a simple
twisted homogeneous rack, that is case
\eqref{item:twisted-homogeneous}, is of type D. We first show that
twisted homogeneous racks (THR, for short) of class $(L,t,
\theta)$ are parameterized by twisted conjugacy classes in $L$
with respect to $\theta$, see Proposition
\ref{prop:twisted-ger-par}. For example, assume that $L=\A_n$,
$n\geq 5$, and $\theta=\iota_{(1\;2)}$ (conjugation in $\s_n$).
Then the THR of type $(\A_n,t,\iota_{(1\;2)})$ are parameterized
by the conjugacy classes in $\s_n$ not contained in $\A_n$; this
explain the notation in Table \ref{tab:MainThm:An:noid}. Then we
develop some techniques to check that a twisted homogeneous rack
is of type D, see Section \ref{section:general-techniques};
neither of these results requires the simplicity of $L$. Our main
results are proved in Section \ref{sect:simple}:

\begin{theorem}\label{th:thr-typeD}
    Let $L$ be $\an$, $n\ge 5$, $\theta\in\Aut(L)$, $t \ge 2$ and $\ell \in L$. If  $\cl$ is a twisted
homogeneous rack of class $(L, t, \theta)$ not listed in Tables
\ref{tab:MainThm:An:id}, \ref{tab:MainThm:An:noid}, then $\cl$ is of type D.
\end{theorem}

\begin{theorem}\label{th:thr-typeD-espo}
Let $L$ be a sporadic group, $\theta= \id$, $t \ge 2$ and $\ell \in L$. If  $\cl$ is a twisted
homogeneous rack of class $(L, t, \theta)$ not listed in Table
\ref{tab:MainThm:espo:id}, then $\cl$ is of type D.
\end{theorem}

The case: $L$ a sporadic group, $\theta\neq \id$, will be treated in another paper.
See however a preliminary discussion in Subsection \ref{subsect:spor-no-id}.

\begin{table}[th]
\begin{center}
\caption{THR $\cl$ of type $(\A_n,t,\theta)$, $\theta=\id$,
$t\geq2$, $n\geq 5$, not known of type D. Those not of type D are
in bold.}\label{tab:MainThm:An:id}
\begin{tabular}{|c|c|c|c|} 
\hline  $n$ & $\ell$ & Type of $\ell$ & $t$ \\

\hline   any & $e$ &$(1^n)$ & odd, $(t,n!)=1$\\
   5 & & ${\bf (1^5)}$ & {\bf 2}\\
   5 & & $(1^5)$ & 4\\

   6 & & ${\bf (1^6)}$ & {\bf 2}\\

\hline
5 & involution & $(1,2^{2})$ & 4, odd\\

6 & & $(1^{2},2^{2})$ & odd\\

8 & &$(2^{4})$ & odd\\

\hline
any & order 4& $(1^{r_1},2^{r_2},4^{r_4})$, $r_4>0$, $r_2+r_4$ even & 2\\

\hline
\end{tabular}
\end{center}
\end{table}

\begin{table}[th]
\begin{center}
\caption{THR $\cl$ of type $(\A_n,t,\theta)$,
$\theta=\iota_{(1\,2)}$, $t\geq 2$, $n\geq 5$, not known of type
D.}\label{tab:MainThm:An:noid}
\begin{tabular}{|c|c|c|} 
\hline  $n$ & Type of $\ell(1\;2)$ & $t$ \\

\hline any & $(1^{s_1},2^{s_2},\dots,n^{s_n})$, $s_1\leq 1$ and $s_2=0$ &  any\\

 & $s_h\geq 1$, for some $h$, $3\leq h\leq n$ & \\

&&\\

 & $(1^{s_1},2^{s_2},4^{s_4})$, $s_1\leq 2$ or $s_2\geq 1$, &  2\\

& $s_2+s_4$ odd, $s_4\geq 1$&\\

\hline
5 &$(1^3,2)$ &  2, 4\\

\hline 6 & $(1^4,2)$ &  2 \\
& $(2^3)$  & 2 \\

\hline 7 & $(1,2^3)$ &  2, odd \\

\hline 8 & $(1^2,2^3)$ &   odd\\

\hline 10 & $(2^5)$ &  odd\\

\hline
\end{tabular}
\end{center}
\end{table}

\begin{table}[ht]
\begin{center}
\caption{THR $\cl$ of type $(L,t,\theta)$, with $L$ a sporadic
group, $\theta=\id$, not known of type
D.}\label{tab:MainThm:espo:id}
\begin{tabular}{|c|c|c|}

\hline  sporadic group & Type of $\ell$ or  & $t$ \\

 &  class name of $\oc_\ell^L$ &  \\

\hline  any &  \textup{1A} &  $(t, |L|)= 1$,  $t$ odd \\

\cline{2-3} &  $\ord(\ell)=4$  & 2\\

\hline $T$, $J_2$, $Fi_{22}$, $Fi_{23}$, $Co_2$  & \textup{2A}  & odd  \\

\hline $B$ & \textup{2A, 2C}  & odd\\

\hline $Suz$ & \textup{6B, 6C}  & any \\

\hline
\end{tabular}
\end{center}
\end{table}

\section{Preliminaries}\label{section:preliminaries}

\subsection{Glossary}\label{subsection:glossary}
If $t, n \in \N$, then $(t,n)$ denotes their highest common
divisor. If $G$ is a group and $x\in G$, then $\langle x\rangle$
denotes the cyclic subgroup generated by $x$; $\iota_x$, the inner
automorphism associated to $x$; $\oc_x^G$, the conjugacy class of
$x$, and $\Cent_G(x)$ its centralizer. We say that $x$ is an
involution if it has order 2. If $u\in \Aut(G)$, then $G^u$
denotes the subgroup of points fixed by $u$. The trivial element
of $G$ is denoted by $e$. A decomposition of a rack $X$ is a
presentation as a disjoint union of two subracks $X = Y \coprod
Z$. A rack is indecomposable if it admits no decomposition.

\subsection{Affine racks of type D}\label{subsection:afines}

Let $ A $ be a finite abelian group and $g\in \Aut (A)$. We denote
by $(A,g)$ the rack with underlying set $A$ and rack
multiplication $x\trid y := g(y) + (\id-g)(x)$, $x, y\in A$; this
is a subrack of the group $A \rtimes \langle g\rangle$. Any rack
isomorphic to some $(A,g)$ is called \emph{affine}.

Notice that $(A, g)$ is indecomposable if and only if $\id-g$ is
invertible. For, assume that $\Imm (\id - g)\neq A$; since $x\trid
y = g(y) + (\id-g)(x) = y + (\id-g) (x-y)$, the decomposition of
$A$ in cosets with respect to $\Imm (\id - g)$ is a decomposition
in the sense of racks. In fact, if $Y$ is any coset, then $A\trid
Y = Y$; thus, the union of any two different cosets is a
decomposable subrack of $A$. The converse of the claim is
\cite[Remark 3.13]{AFGV}.

\medbreak For instance, consider the cyclic group $A = \Z/n$ and
the automorphism $g$ given by the inversion; the rack $(A,g)$ is
denoted $\D_{n}$ and called a dihedral rack. Thus, a family
$(\mu_{i})_{i\in \Z/n}$ of distinct elements of a rack $X$ is
isomorphic to $\D_{n}$ if $\mu_{i} \trid \mu_{j} = \mu_{2i-j}$ for
all $i,j$.

\begin{lema}\label{lema:d2m}
If $m>2$, then the dihedral rack $\D_{2m}$ is of type D.
\end{lema}
\begin{proof}
Since $\id -g = 2$ is not invertible,  $X=\D_{2m}=\{1,2,3,...,2m\}$ is decomposable.
Indeed, if $Y=\{1,3,5,...,2m-1\}$, then $X\triangleright Y\subseteq Y$ and therefore $X=Y\coprod (X\backslash Y)$
is a decomposition. Let $r=1\in Y$ and $s=2\in X\backslash Y$. Then
$r\triangleright(s\triangleright(r\triangleright s))\ne s$,
since $r\triangleright(s\triangleright(r\triangleright s))=-2+2m$
and $m>2$.
\end{proof}

\medbreak We now consider a generalization of this example.
Let $k,t \in \N$, $k>1$, $t >2$
 and consider the affine rack
$ (A,g) $ where $A =(\Z/k)^{t-1}$ and
\begin{equation}\label{eq:subrack-afin}
g(a_{1},\ldots ,a_{t-1})= \left(-\sum_{i=1}^{t-1} a_{i},
a_{1},\ldots, a_{t-2}\right).
\end{equation}

Note that the case $t=2$ corresponds to the dihedral rack.

\begin{lema}\label{lema:A-g-tipo-D}
If $(t,k)\ \neq 1$, and $k > 2$ when $t=4$, then $(A,g)$ is of type D.
\end{lema}

\pf
First we show that $(A,g)$ contains at least two
cosets with respect to $\Imm (\id -g)$. Indeed, let $(a,0,\ldots,0) \in A$
such that $a$ does not belong to
the image of $m_{t}:\Z/n\to \Z/n,\ x\mapsto tx$; such
an $a$ exists since $(t,k)\neq 1$. Then $(a,0,\ldots,0) \notin
\Imm (\id -g)$, since otherwise there exists
$(c_{1},\ldots,c_{t-1})$ such that
\begin{align*}
(a,0,\ldots,0) & =
(\id-g)(c_{1},\ldots,c_{t-1}) = (c_{1} + \sum_{i=1}^{t-1} c_{i}, c _{2} - c_{1},
\ldots, c_{t-1} -c_{t-2})
 \end{align*}
which implies that $c_{i} = c_{j}$ and $a = t c_{1}$ in
$\Z/k$, a contradiction.
Now, take $s=(0,\ldots, 0)$ and
$r = (1,0,\ldots, 0)$; then
$$r\trid(s\trid(r\trid s)) = \begin{cases}(2,-2,2,-1,0,\ldots,0), & t\geq 5; \\
(2,-2,2), & t = 4; \\
(1,-2), & t =3. \\
\end{cases}$$
This is different to $s$, by hypothesis \footnote{ If $k = 2$ and
$t=4$,  then $(A,g)$ is not of type D.}. Thus $(A,g)$ is of type
D. \epf

\section{General techniques}\label{section:general-techniques}

\subsection{Twisted conjugacy classes}\label{subsect:tcc}

Let $G$ be a finite group and $u\in \Aut (G)$.
$G$ acts  on itself by
$y\rightharpoonup_u x = y\,x\,u(y^{-1})$, $x,y\in G$. Let $\oc^{G,u}_x$
be the orbit of $x\in G$ by this action; we call it the \emph{twisted conjugacy class} of $x$.
Then $\oc^{G,u}_x$ is a rack with operation
\begin{equation}\label{eqn:twisted-homogeneous}
y\trid_u z = y\,u(z\,y^{-1}), \quad y,z\in \oc^{G,u}_x.
\end{equation}
The verification that this is  a rack is
straightforward, see \cite[Remark 3.8]{AG-adv}. For instance, if
$y = h\rightharpoonup_u x$ and $z = v\rightharpoonup_u x$, then
$y\trid_u z = w\rightharpoonup_u x$, where $ w =
hxu(h^{-1}v)x^{-1}$.
Of course, if $u = \id$, then this is just the rack structure on a conjugacy class.

The set of twisted conjugacy classes for $u$ only depends on its class in $\Out (G)$.
For, if $n\in G$ and $u' = u \iota_n$, where $\iota_n$ is the inner automorphism associated to $n$,
then $R: (G,u) \to (G, u')$, $R(x) = x\,u(n^{-1})$, is an isomorphism of racks. Indeed, if $x,y\in G$, then
\begin{align*}
R(x \trid y) = xu(yx^{-1}n^{-1}) = xu(n^{-1}) u(nyx^{-1}n^{-1}) = R(x) \trid' R(y).
\end{align*}

\begin{lema}\label{lema:esp-tipo-D}
Let $x\in G$ with $\ord(x)= 2m>4$ and $u(x)=x^{-1}$. If
$\langle x\rangle\subseteq \oc^{G,u}_x$, then $\oc^{G,u}_x$ is of type D.
\end{lema}

The hypothesis ``$\langle x\rangle\subseteq \oc^{G,u}_x$'' is
equivalent to the existence of $y\in G$ such that
$y\rightharpoonup_ux =x^{2}$.

\begin{proof}
Since $x^i\trid_u x^j=x^i u(x^{j-i})=x^{2i-j}$, $\langle x\rangle$ is a subrack of $G$ isomorphic to $\D_{2m}$;
hence Lemma \ref{lema:d2m} applies.
\end{proof}

\bigbreak

The following consequence of \eqref{eqn:twisted-homogeneous} helps
the search of twisted homogeneous racks of type D.

\begin{obs}\label{obs:corta-diagonal}
If $x\in G^u$, then $\oc^{G^u}_x$ is a subrack of $\oc^{G,u}_x$.
\end{obs}

Note that, if $u$ is an involution, then
\begin{equation}\label{eqn:s-noid-condition}
\oc^{G,u}_x\cap G^u \neq\emptyset \iff \exists z\in G: zxz =
u(x) \text{ and }\exists y\in G: z = u(y^{-1})y.
\end{equation}

\bigbreak

\subsection{Twisted homogeneous racks}\label{subect:thr}

We fix a finite group  $L$, $t\in \N$, $t > 1$, and $\theta\in
\Aut (L)$. We consider in the rest of this section
\begin{itemize}
  \item $G = L^t$,

  \item $u\in \Aut (G)$, $u(\ell_1,\ldots,\ell_t)=
(\theta(\ell_t),\ell_1,\ldots,\ell_{t-1})$, $\ell_1,\ldots,\ell_t \in L$.
\end{itemize}

\bigbreak

The twisted conjugacy class of $(x_{1},\ldots, x_{t}) \in L^{t}$ will be called a
\emph{twisted homogeneous rack} of class $(L,t, \theta)$ and denoted  $\C_{(x_{1},\ldots, x_{t})}$, to
avoid confusions. It is also useful to denote $ \C_{\ell} := \C_{(e,\ldots , e,\ell)}$,
$\ell\in L$. We first describe the twisted homogeneous racks of class $(L,t, \theta)$.

\begin{prop}\label{prop:twisted-ger-par}
\renewcommand{\theenumi}{\roman{enumi}}\renewcommand{\labelenumi}{(\theenumi)}
\begin{enumerate}

\item If  $(x_{1},\ldots, x_{t}) \in L^{t}$ and $\ell = x_{t}x_{t-1}\cdots x_{2}x_{1}$,
then $\C_{(x_{1},\ldots ,x_{t})} = \C_{\ell}$.

\item  $ \C_{\ell}= \C_k$ iff $k\in \oc^{L, \theta}_{\ell}$; hence $$\C_{\ell}=\{(x_{1},\ldots, x_{t})\in L^{t}:\
x_{t}x_{t-1}\cdots x_{2}x_{1}\in \oc^{L, \theta}_{\ell}\}.$$

\item  There exists a bijection $\varphi$
between the set of twisted conjugacy classes
of $L$ and the set of twisted homogeneous racks of class $(L,t, \theta)$, given by
$\varphi(\oc^{L, \theta}_{\ell}) = \C_{\ell}$.

\item  $\vert \C_{\ell}\vert = |L|^{t-1} |\oc^{L, \theta}_{\ell}|$.
\end{enumerate}
\end{prop}

\pf
(i). Let $u_j = (x_j \dots x_1)^{-1}$; then
\begin{align*}
(u_{1}, u_{2},\ldots, u_{t-1},e)&\rightharpoonup (x_{1},\ldots, x_{t}) \\&=
(u_{1}x_1, u_{2}x_2u_{1}^{-1},\ldots, u_{t-1}x_{t-1}u_{t-2}^{-1},x_{t}u_{t-1}^{-1})\\
&=(e,\ldots,e, x_{t}x_{t-1}\cdots x_{2}x_{1}).
\end{align*}

(ii). Suppose that there exists $(a_{1},\ldots ,a_{t}) \in L^{t}$ such that
$(a_{1},\ldots, a_{t})\rightharpoonup (e,\ldots, e, \ell) =
(e,\ldots, e, k)$. Then
$
a_{t-1}= a_{t-2}= \ldots = a_{2}= a_{1}=  \theta(a_{t}),
$
hence $k = a_t \ell \theta(a_{t}^{-1})$. Conversely, assume that $k = b \ell \theta(b^{-1})$; if
$a_t = b$, $a_{t-1}= a_{t-2}= \ldots = a_{2}= a_{1}=  \theta(b)$, then
$(a_{1},\ldots, a_{t})\rightharpoonup (e,\ldots, e, \ell) =
(e,\ldots, e, k)$. The second claim follows at once from the first and (i).
Now (iii) is immediate.

(iv). Define the map
 $\psi: L^{t-1}\times \oc^{L, \theta}_{\ell} \to \C_{\ell}$
by $$\psi(b_{1},\ldots,b_{t-1},c) = (b_{1},\ldots, b_{t-1},
 \theta^{-1}(b_{t-1}\cdots b_{1})^{-1}c).$$
 It is a  well-defined map by (ii)  and it is clearly injective. Moreover,
 by the proof of (ii) we know that if
 $(b_{1},\ldots, b_{t}) \in \C_{\ell}$, then
 $b_{t} =  \theta^{-1}(b_{t-1}\cdots b_{1})^{-1}c$ for some
 $c\in \oc^{L, \theta}_{\ell}$. Thus
 $(b_{1},\ldots, b_{t}) = \psi(b_{1},\ldots, b_{t-1},c)$ by definition
 and $ \psi $ is surjective, implying that
$\vert \C_{\ell}\vert = |L|^{t-1} |\oc^{L, \theta}_{\ell}|$.
\epf

\subsection{Twisted homogeneous racks intersecting the diagonal}\label{subsect:diagonal}

Clearly,
\begin{align}\label{eqn:puntos-fijos}
G^u  &= \{(a,\ldots,a): a \in L^\theta \},
\\\intertext{hence, by Proposition \ref{prop:twisted-ger-par}}
\cl \cap G^u &\neq \emptyset \iff \exists\, a \in L^\theta: a^t = \ell.
\end{align}
If this happens, then we have an inclusion of racks
$\oc^{L^\theta}_a\hookrightarrow\cl$, see Remark \ref{obs:corta-diagonal}.
In the particular case $\theta = \id$, if there exists $a \in L$ such that
$a^t = \ell$, we have an inclusion of racks $\oc^{L}_a\hookrightarrow\cl$.

\subsection{Affine subracks of twisted homogeneous racks}
\label{subsection:affine-subracks}

\begin{lema}\label{lema:twisted-D-todo-t}
Let $ \ell \in L $. Assume that there exists $x\in L$ such that
\begin{align}
  \label{eq:cond-uno} \theta(\ell x \ell^{-1}) &= x,
\\\label{eq:cond-dos} (k,t) &\neq 1, \qquad \text{where } k= \ord(x),
\\\label{eq:cond-cuatro} k &\neq 2,4, \qquad \text{when }  t=2,
\\\label{eq:cond-tres} k &> 2, \qquad \text{when }  t=4.
\end{align}
Then $\C_{\ell}$ is of type D.
\end{lema}

When $\theta = \id$, \eqref{eq:cond-uno} just means that $x\in \Cent_{L}(\ell)$.

\pf Let $X=\{(x^{a_{1}},x^{a_{2}},\ldots, x^{a_{t-1}},\ell
x^{a_{t}}):\ \sum_{i=1}^{t}a_{i}\equiv 0 \mod k\}$. Then $X$ is a
subrack of $\C_{\ell}$ since $X\subseteq \C_{\ell}$, by
Proposition \ref{prop:twisted-ger-par} (ii), and
\begin{align*}
(x^{a_{1}},\ldots, \ell x^{a_{t}})&\trid
(x^{b_{1}},\ldots, \ell x^{b_{t}})  =\\
& =(x^{a_{1}},\ldots, \ell x^{a_{t}})
u((x^{b_{1}},\ldots, \ell x^{b_{t}})
(x^{-a_{1}},\ldots, x^{-a_{t}}\ell^{-1}))\\
& = (x^{a_{1}},\ldots, \ell x^{a_{t}})
(\theta(\ell x^{b_{t}-a_{t}}\ell^{-1}),\ldots, x^{b_{t-1}-a_{t-1}})\\
&= (x^{a_{1}+b_{t}-a_{t}},\ldots, \ell x^{a_{t}+b_{t-1}-a_{t-1}}).
\end{align*}
Moreover, let $(A,g)$ be the affine rack considered in Lemma \ref{lema:d2m} or Lemma \ref{lema:A-g-tipo-D}.
Then by the calculation above, there exists a rack isomorphism
$$ (A,g) \xrightarrow{\varphi} X,\qquad
(a_{1},\ldots, a_{t-1}) \mapsto (x^{a_{1}},\ldots, x^{a_{t-1}},
\ell x^{-\sum_{i=1}^{t-1} a_{i}}),$$
which implies that $X$ is of type D, by Lemma \ref{lema:d2m} or Lemma \ref{lema:A-g-tipo-D}.
\epf

\begin{cor}\label{cor:twisted-D-todo-t}
Assume that $t \geq 6$ is even. If $ \ell \in L^\theta$, with
$\ord(\ell)$ even, then $\C_{\ell}$ is of type D.
\end{cor}

\pf Take $x = \ell$ and apply Lemma \ref{lema:twisted-D-todo-t}.\epf

\begin{lema}\label{lema:twisted-2-id-D}
 Assume $ t=2 $ and $ \theta=\id $.
Let $ \ell \in L $. Suppose that there exists $x,y, a\in L$ such that
\begin{equation}\label{eqn:hipotesis-extasis}
\begin{aligned}
\ell &= xy = yx, &\quad  x\neq e& \neq y, &\quad x^{2}&=e,\\ a &\in \Cent_{L}(x)\cap \Cent_{L}(y), &  a^4 &\neq e,
& x &\notin \langle a \rangle.
\end{aligned}
\end{equation}
Then $\C_{\ell}$ is of type D.
\end{lema}

\pf
Let $X=\{(x a^{i},a^{-i}y):\ 0\leq i< \ord(a)\}$ and
$Y=\{(a^{i}, a^{-i}\ell):\ 0\leq i< \ord(a)\}$. Both are
subracks of $\C_{\ell}$ by Proposition \ref{prop:twisted-ger-par} (ii)
and the fact that $a \in \Cent_{L}(x)\cap \Cent_{L}(y)$. Since $x\notin \langle a\rangle$, we have that $X\cap Y = \emptyset$. Moreover,
$X\coprod Y$ is a subrack of $\C_{\ell}$ since
\begin{align*}
(xa^{i},a^{-i}y) \trid (a^{j}, a^{-j}\ell) & =
(a^{2i-j}, a^{-2i+j}\ell)\qquad \text{ and }\\
(a^{i},a^{-i}\ell) \trid (xa^{j}, a^{-j}y) & =
(xa^{2i-j}, a^{-2i+j}y).
\end{align*}

Taking $r= (x,y) $ and $s= (a,a^{-1}\ell)$ we get that $r\trid
(s\trid(r\trid s)) = (a^{-3}, a^3\ell)$. Since $a^4 \neq e$ by
assumption, it follows that $\C_{\ell}$ is of type D. \epf

\begin{cor}\label{cor:twisted-2-id-D}
 Assume $ t=2 $ and $ \theta=\id $.
Let $ \ell \in L$. If $\ord(\ell) = 2m$ with $m$ odd, then
$\C_{\ell}$ is of type D.
\end{cor}

\pf The elements $x = \ell^m$, $y = a = \ell^{1-m}$ satisfy $\ord (y) = m$ and \eqref{eqn:hipotesis-extasis}.
\epf

\bigbreak

Denote by $\mathbb D_{n}= \langle x,y\mid
x^2=1=y^n,\;xyx=y^{-1}\rangle$ the dihedral group of $2n$
elements.

\begin{lema}\label{lema:lea}
Assume $ t=2 $ and $ \theta=\id $. If there exists $\psi:\mathbb
D_{n}\to L$ a group monomorphism, with $n\geq 3$ odd, then
$\mathcal C_{\psi(x)}$ is of type D.
\end{lema}

\pf Let $z=\psi(y)$, $b_1=\psi(x)$ and $b_j=b_1z^{j-1}$, $2\leq j
\leq n$. Then $z^{i}b_{j}=b_{j-i}$ and $b_{i}z^{j}=b_{i+j}$, for
all $1\leq i,j \leq n$.

Let $R = \{(z^{i},z^{-i}b_{j})\mid1\leq i,j\leq n\}$ and $S =
\{(z^{-k}b_{l},z^{k})\mid1\leq k,l\leq n\}$. They are disjoint
since otherwise there exist $i$, $k$, $l\in\{1,\dots,n\}$ such
that $z^i=z^{-k}b_l$; this implies that $z^{i+k}=b_l$, which is a
contradiction because $b_l$ is an involution and $z$ has odd
order. Note that $(z^i,z^{-i}b_{j})=(z^i,b_{j+i})$ and
$(z^{-k}b_{l},z^{k})=(b_{l+k},z^{k})$. Hence $R$ and $S$ are
subracks of $\mathcal C_{\psi(x)}$ since
\begin{align*}
(z^i,b_{j+i})\triangleright (z^k,b_{l+k})& = (z^{2i+j-k-l},
b_{k+j}),\\
(b_{j+i},z^{i})\triangleright (b_{l+k},z^k)&
=(b_{k+j}, z^{2i+j-k-l}).
\end{align*}
Moreover, $T=R\coprod S$ is a decomposable subrack of $\mathcal
C_{\psi(x)}$ because
\begin{align*}
(z^i,b_{j+i})\triangleright (b_{l+k},z^k)& =
(b_{j-k},
z^{-(2i+j-k-l)}),\\
(b_{l+k},x^k)\triangleright (z^i,b_{j+i})& =(z^{-(2k+l-j-i)},
b_{j-k}).
\end{align*}
If we take $r=(1,b_1)$ and s=$(b_2,1)$, then $r\trid (s\trid
(r\trid s))=(b_1b_2b_1,b_1b_2b_1b_2)\neq s$. Therefore, $\mathcal
C_{\psi(x)}$ is of type D. \epf

\bigbreak

\subsection{Twisted homogeneous racks with quasi-real $\ell$}
\label{subsection:quasi-real}

Let $ \ell \in L$ and $j \in \N$. Recall that $\ell$, or $\oc^{L}_{\ell}$, is quasi-real of type $j$
if $\ell^{j}\neq \ell$ and $\ell^{j} \in \oc_{\ell}^{L}$.
For example, if $\ell$ is real, that is $\ell^{-1} \in \oc_{\ell}^{L}$, but not an involution, then
it is quasi-real of type $\ord(\ell) - 1$.

We partially extend this notion to twisted conjugacy classes. We
shall say that $\ell\in L^\theta$ is \emph{$\theta$-quasi-real of
type $j \in \N$} if $\ell^{j}\neq \ell$ and $\ell^{j} \in
\oc_{\ell}^{L, \theta}$. Note that if $\ell\in L^\theta$ is
quasi-real of type $j$ in $L^\theta$, then it is
$\theta$-quasi-real, but the converse is not true. Indeed, if we
take $L=\A_{11}$, $\theta=\iota_{(1\;2)}$ the conjugation by
$(1\,2)$, $\ell=(1\;2)(3\;4)(5\;6\;7\;8\;9\;10\;11)$ and
$y=(1\;4)(2\;3)(5\;6\;8)(7\;10\;9)$, then $y
\rightharpoonup_\theta \ell=\ell^2$, but
$\ell^2\not\in\oc_\ell^{\A_{11}^{\theta}}$.

\begin{lema}\label{lema:twisted-D-todo-t-2}
Assume that $\ell\in L^\theta$ is $\theta$-quasi-real of type $j$.
If $t\geq 3$ or $t=2$ and $\ord(\ell) \nmid 2(1-j)$, then
$\C_{\ell}$ is of type D.
\end{lema}

\pf
Let
\begin{align*}
X &= \{(\ell^{a_{1}}, \ldots, \ell^{a_{t}}):\ \textstyle \sum_{i=1}^{t}a_{i}=j \},\\
Y &= \{(\ell^{a_{1}}, \ldots, \ell^{a_{t}}):\ \textstyle \sum_{i=1}^{t}a_{i}=1 \}.
\end{align*}
By definition, $X\cap Y = \emptyset$.
We have
\begin{equation}\label{eq:trid}
(\ell^{a_{1}}, \ldots, \ell^{a_{t}})
\trid (\ell^{b_{1}}, \ldots, \ell^{b_{t}}) =
(\ell^{a_{1}+b_{t}-a_{t}}, \ell^{a_{2}+b_{1}-a_{1}},
\ldots, \ell^{a_{t}+b_{t-1}-a_{t-1}}).
\end{equation}
Hence $X$ and $Y$ are subracks of $\C_{\ell}$,
$X\trid Y \subseteq Y$, $Y\trid X \subseteq X$, and $X\coprod Y$ is
a subrack of $\C_{\ell}$.

\medbreak
Take $r= (1,\ldots, 1, \ell^{j}) \in X$ and $s= (1,\ldots, 1, \ell)\in Y$.
We show that $r\trid (s\trid (r\trid s)) \neq s$.
Assume first that $t\geq 4$. Then by
\eqref{eq:trid} we have that
\begin{align*}
r\trid (s\trid (r\trid s)) & =
r\trid (s\trid (\ell^{1-j}, 1, \ldots, 1, \ell^{j})) =
r\trid  (\ell^{j-1},\ell^{1-j}, 1, \ldots, 1, \ell)\\
& = (\ell^{1-j}, \ell^{j-1},\ell^{1-j}, 1, \ldots, 1, \ell^{j})  \neq
  (1,\ldots, 1, \ell).
\end{align*}
If $t=3$, then $r=(1,1,\ell^{j})$, $s= (1,1,\ell)$ and the
computation yields
\begin{align*}
r\trid (s\trid (r\trid s)) & =
r\trid (s\trid (\ell^{1-j},1, \ell^{j})) =
r\trid  (\ell^{j-1},\ell^{-j+1},\ell)\\
&= (\ell^{1-j},\ell^{j-1}, \ell)\neq (1,1,\ell).
\end{align*}
Finally if $t=2$, then $r=(1,\ell^{j})$, $s= (1,\ell)$
and we have
\begin{align*}
r\trid (s\trid (r\trid s)) & =
r\trid (s\trid (\ell^{1-j},  \ell^{j})) =
r\trid  (\ell^{j-1},\ell^{-j+2})
= (\ell^{-2j+2}, \ell^{2j-1}).
\end{align*}
Thus, $\C_{\ell}$ is of type D if $(\ell^{-2j+2}, \ell^{2j-1})
\neq (1,\ell)$ which amounts to $\ord (\ell)   \nmid 2(1-j)$. \epf

\section{Simple twisted homogeneous racks}\label{sect:simple}

Summarizing the results of the previous section, and with the same
notation, we have
\begin{prop}\label{prop:summary}
\renewcommand{\theenumi}{\roman{enumi}}\renewcommand{\labelenumi}{(\theenumi)}
Let $\ell \in L^\theta$.
\begin{enumerate}
\item If $\ell \in L$ is quasi-real of type $j$,
$t\geq 3$ or $t=2$ and $\ord(\ell) \nmid 2(1-j)$, then $\cl$ is of type D.

\item If $\ord(\ell)$ is even and $t \geq 6$ is even, then $\cl$ is of type D.

\item If $\ell$ is an involution, $t$ is odd and
$\oc^{L^\theta}_\ell$ is of type D, then so is $\cl$.

\item If $\ell$ is an involution, $t =4$ and there exists $x\in C_{L^\theta}(\ell)$ with $\ord(x) = 2m > 2$, $m\in \N$,
then $\C_\ell$ is of type D.

\item If $\ell$ is an involution, $t = 2$,  and there exists $x\in C_{L^\theta}(\ell)$ with $\ord(x) = 2m > 4$, $m\in \N$,
then $\C_\ell$ is of type D.

\item If $\ell$ is an involution, $t = 2$, and there exists $\psi:\mathbb D_{n}\to
L^{\theta}$ a group monomorphism, with $n\geq 3$ and
$\ell=\psi(x)$ for some $x\in\dn$ involution, then $\mathcal
C_{\ell}$ is of type D.

\item If $(t, \vert L^\theta\vert)$ is divisible by an odd prime $p$, then $\C_e$ is of type D.

\item If $(t, \vert L^\theta\vert)$ is divisible by $p=2$ and $t \geq 6$, then $\C_e$ is of type D.

\item If $t =4$ and there exists $x\in L^\theta$ with $\ord(x) = 2m > 2$, $m\in \N$,
then $\C_e$ is of type D.

\item If $t = 2$,  and there exists $x\in L^\theta$ with $\ord(x) = 2m > 4$, $m\in \N$,
then $\C_e$ is of type D.
\end{enumerate}
\end{prop}

\pf (i) is Lemma \ref{lema:twisted-D-todo-t-2}; (ii) is Corollary
\ref{cor:twisted-D-todo-t}; (iii) follows from the discussion in
Subsection \ref{subsect:diagonal}, since $(\ell, \dots, \ell) \in
\cl$. Case (vi) follows by Lemma \ref{lema:lea}. Finally, the
remaining cases follow from Lemma \ref{lema:twisted-D-todo-t}: in
cases (vii) and (viii), take $x$ of order $p$, and in cases (iv),
(v), (ix) and (x) the given $x$.
\epf

We now explore when the proposition applies to a simple twisted
homogeneous rack, that is when $L$ is simple.

\begin{obs}\label{obs:simple-par}
If $L$ is simple non-abelian and $\theta = 1$, $t\neq 4,2$ and
$(t, \vert L\vert)\neq 1$, then $\C_e$ is of type D. This follows
from Proposition \ref{prop:summary} (vii) and (viii).
\end{obs}

\subsection{$L = \mathbb{A}_{n}$, $n\ge 5$, $\theta = \id$}
\label{subsect:An-id}

\bigbreak

In this subsection we prove our main Theorem \ref{th:thr-typeD} in
this case by applying all the results obtained above.

\subsubsection{$\ell = e$}\label{subsubsec:e} We now treat the
Table \ref{tab:MainThm:An:id}.
\begin{itemize}
 \item If $t\geq 6$ even or $t$ odd with $(t, n!)\neq 1$, then $\C_e$ is of type
 D, by Lemma \ref{lema:twisted-D-todo-t}.

\item Assume that $t=4$. If $n\geq 6$, then $\C_e$ is
of type D, by Proposition \ref{prop:summary} (ix) with
$x=(1\;2)(3\;4\;5\;6)$. If $n=5$, we do not know if $\C_e$ is
of type D.

\item Assume that $t=2$. If $n \geq 7$, then $\C_e$ is of
type D, by Proposition \ref{prop:summary} (x) with $x
=(1\;2)(3\;4)(5\;6\;7)$. If $n=5$ or $6$, then Proposition
\ref{prop:summary} (x) does not apply because the only
possible even orders of elements are 2 and 4. Moreover, we
have checked that $\C_e$ is not of type D, using \textsf{GAP}.
\end{itemize}

\subsubsection{$\ell$ an involution}\label{subsubsec:involuciones}
The type of $\ell$ is $(1^{r_1},2^{r_2})$, with $n=r_1+2r_2$ and
$r_2$ even. Assume that $\ell = (i_1\;i_2)(i_3\;i_4)\cdots
(i_{2r_2-1}\;i_{2r_2})$. Then
\begin{itemize}
  \item By \cite[Thm. 4.1]{AFGV},
  $\oc^{L}_\ell$ is of type D,
  except for the following cases:
\begin{align*}
(1, 2^2), \,\, n =5; \qquad (1^2, 2^2), \,\, n =6; \qquad(2^4), \,\,n =8.
\end{align*}
In particular, $\C_{\ell}$ is of type D for all $t$ odd,
except for the cases above.

\item If $t \geq 6$ is even, then $\cl$ is of type D,
by Proposition \ref{prop:summary} (ii).

\item Assume that $t=4$. If $r_2\geq 4$, then
$\C_{\ell}$ is of type D by Proposition \ref{prop:summary}
(iv) with
\begin{align}\label{eq:x}
x=(i_1\;i_2)(i_3\;i_5\;i_7\;i_4\;i_6\;i_8).
\end{align}
Suppose that $r_2= 2$. If $r_1\geq 2$, then $\C_{\ell}$ is of
type D by Proposition \ref{prop:summary} (iv) with
$x=(i_1\;i_3\;i_2\;i_4)(j_1\;j_2)$, where $j_1$, $j_2$ are
fixed by $\ell$. In the case $r_1=1$, Proposition
\ref{prop:summary} (iv) does not apply because the only
possible even order of elements is 2.

\medbreak

\item Assume that $t=2$. If $r_2\geq 4$, then
$\C_{\ell}$ is of type D by Proposition \ref{prop:summary} (v)
with $x$ as in \eqref{eq:x}. Suppose that $r_2= 2$. If
$r_1\geq 3$, then $\C_{\ell}$ is of type D by Proposition
\ref{prop:summary} (v) with
$x=(i_1\;i_2)(i_3\;i_4)(j_1\;j_2\;j_3)$, where $j_1$, $j_2$,
$j_3$ are fixed by $\ell$. In the case $r_1=1$ or $2$,
$\C_\ell$ yields of type D, by Lemma \ref{lema:lea} taking
$\psi:\mathbb D_3\simeq\langle x:=(1\; 2) (3\;4), (1\; 2)
(3\;5)\rangle \hookrightarrow L$.

\end{itemize}

\subsubsection{$\ord(\ell) > 2$}\label{subsubsec:otras}
It is known that:
\begin{equation}\label{eq:quasirealAn}
\begin{aligned}
&\text{$\oc^{L}_\ell$ is quasi-real of type 4, if $\ord(\ell)  > 3$ is odd,}\\
&\text{$\oc^{L}_\ell$ is quasi-real of type $\ord(\ell) - 1$, if
$\ord(\ell)$ is even or 3.}
\end{aligned}
\end{equation}

Now, if $t> 2$ or $t=2$ and $\ord(\ell)\neq 4$, then $\C_{\ell}$
is of type D, by Proposition \ref{prop:summary} (i). On the other
hand, if $t=2$ and $\ord(\ell)= 4$, then Proposition
\ref{prop:summary} (i) does not apply and we do not know if $\cl$
is of type D.


\subsection{$L = \mathbb{A}_{n}$, $n\ge 5$,  $\theta \neq \id$}
\label{subsect:An-no-id}

\bigbreak

We may suppose that $\theta$ is the inner automorphism associated
with an element $\sigma$ in $\s_n -\A_n$. We choose
$\sigma=(1\;2)$; thus, $\theta(\ell)=(1\;2)\,\ell\,(1\;2)$, for
all $\ell\in \A_n$. Hence,
\begin{align}\label{eq:An^theta}
\A_n^\theta =\big(\{(1\;2)\} \times
(\s_{\{3,...,n\}}-\A_{\{3,...,n\}})\big) \coprod
\A_{\{3,...,n\}}
\end{align}
and the order of $L^{\theta}$ is $(n-2)!$.

Let $\ell\in\A_n$. Recall that $\C_{\kappa}=\C_{\ell}$ if and only
if $\kappa\in\oc_{\ell}^{\A_n,\theta}$, by Proposition
\ref{prop:twisted-ger-par} (ii), and notice that
$\kappa\in\oc_{\ell}^{\A_n,\theta}$ amounts to $\kappa(1\;2)\in
\oc_{\ell(1\;2)}^{\s_n}$, the conjugacy class of $\ell(1\;2)$ in
$\s_n$. Thus, we will proceed with our study according to the type
$(1^{s_1},2^{s_2},\dots,n^{s_n})$ of $\ell (1\;2)$.

Notice that $\ord(\ell (1\;2))$ is even, since $\ell
(1\;2)\in\s_n-\A_n$, and $\ord(\ell (1\;2))\geq 4$ if and only if
$s_h\geq 1$, for some $h\geq 4$, or $s_2\geq 1$ and $s_3\geq 1$.

\medbreak

Now, we want to give a description analogous to
\eqref{eq:quasirealAn} for the case $\theta \neq \id$. First, we
note that $\oc_\ell^{\A_n,\theta}\cap \A_n^{\theta}\neq \emptyset$
if and only if $s_1\geq 2$ or $s_2\geq 1$. Indeed, this follows
from the form of $\A_n^{\theta}$ described in \eqref{eq:An^theta}.

Assume that $s_1\geq 2$ or $s_2\geq 1$ and let
$\kappa\in\oc_\ell^{\A_n,\theta}\cap \A_n^{\theta}$. Since
$\oc_{\ell}^{\A_n,\theta}=\oc_{\ell(1\;2)}^{\s_n}\cdot (1\;2)$, it
is clear that
\begin{align}\label{eq:relTCC-CC}
\text{$\kappa$ is $\theta$-quasi-real of type $j$ if and only if
$\kappa^j(1\;2)\in \oc_{\ell(1\;2)}^{\s_n}$.}
\end{align}
If $\ord(\ell(1\;2))\geq 4$, then $\kappa(1\;2)$ is quasi-real of
type $j$ in $\s_n$ with $j=\ord(\ell(1\;2))-1$, by
\eqref{eq:quasirealAn}. Thus, $\kappa^j(1\;2)=(\kappa(1\;2))^j$
and $\kappa^j(1\;2)\in\oc_{\ell(1\;2)}^{\s_n}$. Hence, $\kappa$ is
$\theta$-quasi-real of type $j$, by \eqref{eq:relTCC-CC};
moreover, $\kappa$ yields quasi-real of type $j$ in $\A_n^\theta$.
On the other hand, if $\ord(\ell(1\;2))= 2$, then
$\kappa\in\oc_{\ell(1\;2)}^{\s_n}$ is $\theta$-quasi-real (of type
$j=0$) if and only if $\oc_{\ell}^{\A_n,\theta}$ contains to $e$,
i.~e. $\oc_{\ell(1\;2)}^{\s_n}$ is the conjugacy class of the
transpositions in $\s_n$.

\medbreak

We study now the twisted homogeneous racks $(\mathbb{A}_{n},
t,\theta)$, $n\ge 5$ and $\theta\neq \id$, for different $t's$
according to the type $(1^{s_1},2^{s_2},\dots,n^{s_n})$ of $\ell
(1\;2)$. If $s_1\leq 1$ and $s_2=0$, then
$\oc_\ell^{\A_n,\theta}\cap \A_n^{\theta}=\emptyset$, and we do
not know if $\C_{\ell}$ is of type D.

\emph{From now on we will assume that $s_1\geq 2$ or $s_2\geq 1$
and let $\kappa\in\oc_\ell^{\A_n,\theta}\cap \A_n^{\theta}$.}
Notice that $\C_{\ell}=\C_{\kappa}$ and that $\ell(1\;2)$ and
$\kappa(1\;2)$ have the same type, hence the same order. Moreover,
if we denote the type of $\kappa$ by
$(1^{r_1},2^{r_2},\dots,n^{r_n})$, then $r_h=s_h$, for all $h$,
$3\leq h\leq n$; thus, $\ord(\kappa)$ divides $\ord(\ell(1\;2))$.

We will consider different cases. If $s_h\geq 1$ for
 $h=3$ or $5\leq h\leq n$, then $\ord(\ell(1\;2))> 4$ and
$\cl$ is of type D for all $t$, by Proposition \ref{prop:summary}
(i). Assume that $s_h=0$ for $h=3$ and $5\leq h\leq n$. Suppose
that $s_4\geq 1$; thus $\ord(\ell(1\;2))= 4$. If $t\geq 3$, then
$\cl$ is of type D for all $t$, by Proposition \ref{prop:summary}
(i); whereas if $t=2$, we do not know if $\cl$ is of type D.
Assume that $s_4=0$. Thus, the type of $\ell(1\;2)$ is
$(1^{s_1},2^{s_2})$, with $s_1\geq 2$ or $s_2\geq 1$ odd.

Suppose that $s_2=1$; thus, $e\in\oc_\ell^{\A_n,\theta}$. If
$t\geq 3$, then $\cl=\C_e$ is of type D for all $t$, by
Proposition \ref{prop:summary} (i); whereas if $t=2$, $n\geq 7$
and take $\kappa=e$, then $\C_{\ell}=\C_e$ is of type D, by
Proposition \ref{prop:summary} (x) choosing
$x=(1\;2)(3\;4)(5\;6\;7)$.

Suppose that $s_2>1$ odd. Then the type of
$\kappa\in\oc_\ell^{\A_n,\theta}\cap \A_n^{\theta}$ is
$(1^{r_1},2^{r_2})$, with $n=r_1+2r_2$ and $r_2\geq 2$ even, i.~e.
$\kappa = (i_1\;i_2)(i_3\;i_4)\cdots (i_{2r_2-1}\;i_{2r_2})$. We
determine now when $\oc_{\kappa}^{\A_n^{\theta}}$ is of type D. We
have two possibilities.
\begin{itemize}
\item[(i)] Assume that $\kappa$ fixes 1 and 2. If the type
$(1^{r_1-2},2^{r_2})$ is distinct to $(1,2^2)$, $(1^2,2^2)$
and $(2^4)$, then $\oc_\kappa^{\A_n^{\theta}}$ is of type D,
by \cite[Thm. 4.1]{AFGV}; otherwise,
$\oc_\kappa^{\A_n^{\theta}}$ is not of type D.

\item[(ii)] Assume that $\kappa$ does not fix 1 nor 2; thus $\kappa=(1 2)(i_3 i_4)\cdots (i_{2r_2-1} i_{2r_2})$. 
If the type $(1^{r_1},2^{r_2-1})$ is distinct to $(2^3)$ and
$(1^{r_1},2)$, then $\oc_\kappa^{\A_n^{\theta}}$ is of type D,
by \cite[Thm. 4.1]{AFGV}; otherwise,
$\oc_\kappa^{\A_n^{\theta}}$ is not of type D.
\end{itemize}

\bigbreak

We consider now different values of $t$.

\begin{itemize}

 \item Assume that $t$ is odd. If $\kappa$ fixes 1 and 2 and the type of $\kappa$ is distinct to
$(1^3,2^2)$, $(1^4,2^2)$ and $(1^2,2^4)$, then $\C_{\ell}$ is
of type D, by (i) above and Proposition \ref{prop:summary}
(iii). On the other hand, if $\kappa$ does not fix 1 nor 2 and
the type of $\kappa$ is distinct to $(2^4)$ and
$(1^{r_1},2^2)$, for any $r_1$, then $\C_{\ell}$ is of type D,
by (ii) above and Proposition \ref{prop:summary} (iii).

\medbreak

\item If $t\geq 6$ even, then $\C_{\ell}$ is of type D, by Proposition
\ref{prop:summary} (ii).

\medbreak

\item Assume that $t=4$. We will determine when there exists $x\in
\A_n^\theta$ such that $\ord(x)\geq 4$ even and $\theta(\kappa
x\kappa)=x$, i.~e. $(1\;2)\kappa x\kappa (1\;2) =x$. If
$\kappa$ fixes 1 and 2, take $x=(1\;2)(i_1\;i_3\;i_2\;i_4)$.
If $\kappa(1)=2$ and $\kappa(2)=1$, take
$x=(1\;2)(i_3\;i_5\;i_4\;i_6)$ when $r_2\geq 4$ and
$x=(1\;j_1\;2\;j_2)(i_3\;i_4)$ when $r_2=2$ and $r_1\geq 2$,
$j_1$, $j_2$ are fixed by $\kappa$. In all these cases, $\cl$
is of type D, by Lemma \ref{lema:twisted-D-todo-t}. For the
remaining cases we do not know if $\cl$ is of type D.

\medbreak

\item Assume that $t=2$. We will determine when there exists $x\in
\A_n^\theta$ such that $\ord(x)\geq 6$ even and $\theta(\kappa
x\kappa)=x$, i.~e. $(1\;2)\kappa x\kappa (1\;2) =x$. If
$\kappa$ fixes 1 and 2, take
$x=(1\;2)(i_3\;i_5\;i_7\;i_4\;i_6\;i_8)$ when $r_2\geq 4$ and
$x=(1\;i_1\;i_3\;2\;i_2\;i_4)(j_1\;j_2)$ when $r_2=2$ and
$r_1\geq 4$, $j_1$, $j_2$ are fixed by $\kappa$. If
$\kappa(1)=2$ and $\kappa(2)=1$, take
$x=(1\;2)(i_3\;i_5\;i_7\;i_4\;i_6\;i_8)$ when $r_2\geq 4$ and
$x=(1\;2)(i_3\;i_4)(j_1\;j_2\;j_3)$ when $r_2=2$ and $r_1\geq
3$, $j_1$, $j_2$, $j_3$ are fixed by $\kappa$.
\end{itemize}

\bigbreak

Therefore the cases where $\cl$ is not known to be of type D are
the following
\begin{itemize}
\item[(a)] $\ell(1\;2)$ of type $(1^{s_1},2^{s_2},\dots,n^{s_n})$, $s_1\leq 1$ and $s_2=0$ for any
$t$;
\item[(b)] $\ell(1\;2)$ of type $(1^{s_1},2^{s_2},4^{s_4})$, $s_1\leq 2$ or $s_2\geq 1$, $s_2+s_4$ odd, $s_4\geq 1$ and
$t=2$;
\end{itemize}
and those in Table \ref{tab:An:noid}. Finally, from (a), (b) and
Table \ref{tab:An:noid} we obtain Table \ref{tab:MainThm:An:noid}.

\begin{table}[th]
\begin{center}
\caption{}\label{tab:An:noid}
\begin{tabular}{|c|c|c|c|} 
\hline  $n$ & Type of $\ell$ & $\ell$ & $t$ \\

\hline 6 & $(1^2,2^2)$  & involution & 2 \\

7 & $(1^3,2^2)$ & fixing 1 and 2 & 2, odd \\ 

8 & $(1^4,2^2)$ &  & odd\\ 

10 & $(1^2,2^4)$ &  & odd\\ 

\hline

 5 & $(1,2^2)$  & involution & 2, 4 \\
 6 & $(1^2,2^2)$  & permuting 1 and 2 & 2 \\

8  &  $(2^4)$  &  & odd \\ 

\hline
\end{tabular}
\end{center}
\end{table}

\subsection{$L = $ sporadic group, $\theta = \id$}\label{subsect:spor-id}

In this subsection we prove our main Theorem
\ref{th:thr-typeD-espo}.

\subsubsection{$\ell = e$}\label{subsubsec:e:espo}

\begin{itemize}

  \item If $(t, |L|)\neq 1$, with $t$ odd or $t\geq 6$ even, then $\C_e$ is of type D; see
Table \ref{tab:spor:primos} for the prime numbers dividing the
order of a sporadic group. In particular, if $t\geq 6$ even,
then $\C_e$ is of type D since $|L|$ is even.


\begin{table}[ht]
\begin{center}
\caption{Prime divisors of orders of sporadic
groups.}\label{tab:spor:primos}
\begin{tabular}{|p{1.69cm}|p{3.79cm}||p{1cm}|p{4.5cm}|}
\hline $L$ & {\bf Prime divisors} & $L$ & {\bf Prime divisors}
\\ \hline  $M_{11}$, $M_{12}$ &  2, 3, 5, 11  &
$Co_{1}$ & 2, 3, 5, 7, 11, 13, 23
\\ \hline   $M_{22}$, $HS$, \newline $McL$ & 2, 3, 5, 7, 11  &
$J_{1}$, \newline $HN$ & 2, 3, 5, 7, 11, 19
\\ \hline  $M_{23}$, $M_{24}$, \newline $Co_{2}$, $Co_{3}$  &  2, 3, 5, 7, 11, 23  &  $O'N$ & 2, 3, 5, 7, 11, 19, 31
\\ \hline  $J_{2}$ &  2, 3, 5, 7  &
$J_{3}$ & 2, 3, 5, 17, 19
\\ \hline   $Suz$, $Fi_{22}$ & 2, 3, 5, 7, 11, 13   &    $Ru$ & 2, 3, 5, 7, 13, 29
\\ \hline  $T$ & 2, 3, 5, 13 &  $Fi_{23}$ &  2, 3, 5, 7, 11, 13, 17, 23
\\ \hline  $He$ & 2, 3, 5, 7, 17  &  $Fi'_{24}$ &     2, 3, 5, 7, 11, 13, 17, 23, 29
\\ \hline    $Th$ & 2, 3, 5, 7, 13, 19, 31 &     $B$ &    2, 3, 5, 7, 11, 13, 17, 19, 23, 31, 47
\\ \hline   $J_4$ &      2, 3, 5, 7, 11, 23, 29, 31, 37, 43 &   $M$ &  2, 3, 5, 7, 11, 13, 17, 19, 23, 29, 31, 41, 47, 59, 71
\\ \hline  $Ly$ &   2, 3, 5, 7, 11, 31, 37, 67 &&

\\ \hline
\end{tabular}
\end{center}
\end{table}


  \item If $t=2$ or $t=4$, then $\C_e$ is of type D, by Proposition \ref{prop:summary} (ix) and (x), since always there exists an element $x\in L$
  of order 6. 
\end{itemize}

\subsubsection{$\ell$ an involution}\label{subsubsec:involuciones:espo}

\begin{itemize}

  \item By \cite[Thm. II]{AFGV-espo}, if $\ell$ is an involution then, $\oc^{L}_\ell$ is of type D,
  except for the cases listed in Table \ref{tab:0}; in
particular, $\C_{\ell}$ is of type D for all $t$ odd, except
for these cases.

\begin{table}[ht]
\begin{center}
\caption{Classes of involutions not known of type D;
\newline those which are NOT of type D appear in bold.}\label{tab:0}
\begin{tabular}{|p{2cm}|c||p{1cm}|c||p{1cm}|c|}
\hline $G$ & {\bf Classes} & $G$ & {\bf Classes}& $G$ & {\bf Classes}
\\ \hline  $J_{2}$ &  \textup{\bf 2A} & $Fi_{22}$ &  \textup{{\bf 2A}} & $Co_{2}$ &  \textup{{\bf 2A}}
\\ \hline  $B$ &  \textup{2A, 2C} &   $Fi_{23}$ &  \textup{{\bf 2A}}   &  $T$ &  \textup{2A}
\\ \hline
\end{tabular}
\end{center}
\end{table}

\item If $t \geq 6$ is even, then $\cl$ is of type D,
by Proposition \ref{prop:summary} (ii).

  \item If $t=2$ or $4$, then $\C_{\ell}$ is of type
  D by Proposition \ref{prop:summary} (iv) and (v), since
  always there exists $x\in \Cent_{L}(\ell)$ with $\ord(x)>4$
  even. To see this we use
  \cite{BR,Bo,Br,GAP,Iv,W,AtlasRep1.4}.
\end{itemize}

\subsubsection{$\ord(\ell) > 2$}\label{subsubsec:otras:espo}
The class $\oc^{L}_\ell$ is real or quasi-real, except the classes
\textup{6B, 6C} of the Suzuki group $Suz$.

\begin{itemize}
  \item Assume that $\ell$ do not belong to the class \textup{6B} or \textup{6C} of the Suzuki group
  $Suz$. If $t\geq 3$ or $t=2$ and $\ord(\ell)\neq 4$, then
  $\C_{\ell}$ is of type D, by Proposition \ref{prop:summary}
  (i).

\item Suppose that $\ell$ belongs to the class 6B or  6C of  the Suzuki group $Suz$.
If $t\geq 6$ even, then $\C_{\ell}$ is of type D by
Proposition \ref{prop:summary} (ii); whereas if $t =2$, then
$\C_{\ell}$ is of type D by Corollary
\ref{cor:twisted-2-id-D}.


\end{itemize}

\subsection{$L  =$ sporadic, $\theta \neq \id$}\label{subsect:spor-no-id}

The sporadic groups with non-trivial outer automorphism group are
$M_{12}$, $M_{22}$, $J_2$, $Suz$, $HS$, $McL$, $He$, $Fi_{22}$,
$Fi'_{24}$, $O'N$, $J_3$, $T$ and $HN$. For these groups the outer
automorphism group is $\Z/2$ in all cases. In Table
\ref{tab:L^thetaesp} we give the orders of $L^\theta$ when $L\neq
HN$; we cannot determine the order of $HN^{\theta}$ with our
computational resources. We will assume $L\neq HN$.


\begin{table}[ht]
\begin{center}
\caption{Orders of $L^\theta$.}\label{tab:L^thetaesp}
\begin{tabular}{|c|c||c|c||c|c|}
\hline $L$ & $|L^\theta|$ & $L$ & $|L^\theta|$ & $L$ & $|L^\theta|$\\
\hline $M_{12}$& 120 & $M_{22}$ & 1344& $J_2$  & 336 \\

\hline $Suz$& 1209600 & $HS$ & 40320 & $McL$ & 7920 \\

\hline $He$& 7560 & $Fi_{22}$ & 54 & $Fi'_{24}$ & 4089470473293004800 \\

\hline $T$ & 96   & $O'N$& 175560 & $J_3$  &  2448\\
\hline
\end{tabular}
\end{center}
\end{table}

We describe the case $\ell = e$. The remaining cases will be
treated in a separated paper.
\begin{itemize}
\item If $t\geq 6$ even or $t\geq 3$ odd and $(t,|L^\theta|)\neq 1$,
then $\C_e$ is of type D, by Proposition \ref{prop:summary}
(vii) and (viii). In particular, if $t\geq 6$ even, then
$\C_e$ is of type D since $|L^\theta|$ is even.

 \item Assume that $t=2$ or $t=4$. We have checked
with \textsf{GAP} that there exists $x\in L^\theta$ such that
$\ord(x)>4$ even. Then $\C_e$ is of type D, by Proposition
\ref{prop:summary} (ix) and (x).
\end{itemize}


\begin{thebibliography}{AFGV2}


\bibitem[AFGV1]{AFGV} {N. Andruskiewitsch,} {F. Fantino}, {M.
Gra\~na} {and  L. Vendramin}, \emph{Finite-di\-mensional pointed
Hopf algebras with alternating groups are trivial}, Ann. Mat. Pura
Appl. (4), to appear.


\bibitem[AFGV2]{AFGV-espo} \bysame,
\emph{Pointed Hopf algebras over the sporadic groups}.
Submitted, preprint  \texttt{arXiv:1001.1108}.

\bibitem[AG]{AG-adv} {N. Andruskiewitsch} and {M. Gra\~na},
    \emph{From racks to pointed Hopf algebras}, Adv. Math.
    \textbf{178}  (2003), 177 -- 243.

\bibitem[AHS]{AHS} {N. Andruskiewitsch}, {I. Heckenberger} and {H.-J. Schneider},
\emph{The Nichols algebra of a semisimple Yetter-Drinfeld module},
Amer. J. Math., to appear.



\bibitem[BR]{BR} {C. Bates} and {P. Rowley}, \emph{Involutions in Conway's largest simple
group}, LMS J. Comput. Math. {\bf 7} (2004), 337 -- 351.


\bibitem[Bo]{Bo} {J. van Bon}, \emph{On distance-transitive graphs and involutions},
Graphs Combin. {\bf 7} 4 (1991), 377 -- 394.


\bibitem[Br]{Br}
{T. Breuer}, \emph{The \textsf{GAP} Character Table Library,
Version 1.2 (unpublished)};
\texttt{http://www.math.rwth-aachen.de/\~{}Thomas.Breuer/ctbllib/}


\bibitem[\textsf{GAP}]{GAP}
 The \textsf{GAP}~Group, \emph{\textsf{GAP} -- Groups, Algorithms, and Programming,
 Version 4.4.12};
 2008,
 \verb+(http://www.gap-system.org)+.

\bibitem[HS]{HS1} I. Heckenberger and H.-J. Schneider,
\emph{Root systems and Weyl groupoids for semisimple Nichols
algebras}. Proc. London Math. Soc., to appear.


\bibitem[Iv]{Iv} {A. Ivanov},
\emph{ The fourth Janko group}, Oxford Mathematical Monographs,
The Clarendon Press Oxford University Press, Oxford, 2004.



\bibitem[J]{jo} D. Joyce,
    \emph{Simple quandles}, J. Algebra \textbf{79} 2 (1982), 307 -- 318.

\bibitem[W]{W} {R. A. Wilson}, \emph{Some subgroups of the Baby Monster}, Invent.
Math.  {\bf 89} 1 (1987), 197 -- 218.

\bibitem[W$^+$]{AtlasRep1.4} {R. A. Wilson, R. A. Parker, S. Nickerson, J. N. Bray and T. Breuer},
\emph{{AtlasRep}, A \textsf{GAP} Interface to the \textsf{ATLAS} of
                  Group Representations,
                  {V}ersion 1.4},
\verb+http://www.math.rwth-aachen.de/~Thomas.Breuer/atlasrep+,
2007, Refereed \textsf{GAP} package,


\end{thebibliography}
\end{document}